\documentclass[11pt]{amsart}
\usepackage{fullpage}
\usepackage{graphicx}	
\usepackage{enumitem}			
\usepackage{tikz}

\newtheorem{thm}{Theorem}[section]
\newtheorem{defi}[thm]{Definition}

\newtheorem{conjecture}[thm]{Conjecture}
\newtheorem{problem}[thm]{Problem}

\newtheorem{lemma}[thm]{Lemma}
\newtheorem{cor}[thm]{Corollary}

\makeatletter
\def\@cite#1#2{{\normalfont[{\bfseries#1\if@tempswa , #2\fi}]}}
\makeatother

\DeclareMathOperator{\SN}{SN}
\DeclareMathOperator{\cyc}{cyc}
\DeclareMathOperator{\trn}{trn}

\newcommand{\strong}[3]{\SN^{#2}_{#3}(#1)}
\newcommand{\floor}[1]{\left\lfloor #1 \right\rfloor}

\newcommand{\ivec}{\mathbf{i}_{\mathcal{P}}}
\newcommand{\Ivec}{I_{\mathcal{P}}}
\newcommand{\uIvec}{I^\mu_{\mathcal{P}}}
\newcommand{\uvec}{\mathbf{u}}
\newcommand{\lattice}{L_{\mathcal{P}}}
\newcommand{\ulattice}{L^{\mu}_{\mathcal{P}}}
\newcommand{\uplattice}{L^{\mu}_{\mathcal{P'}}}

\title{Cyclic triangle factors in regular tournaments}

\author{Lina Li}
\address{
  Department of Mathematics,
  University of Illinois at Urbana-Champaign, 
  Urbana, Illinois 61801, USA.
}
\email{linali2@illinois.edu}

\author{Theodore Molla}
\address{
  Department of Mathematics and Statistics,
  University of South Florida, Tampa, Florida 33620, USA. 
  Research is partially supported by NSF Grant DMS-1500121.
} 
\email{molla@usf.edu}

\date{\today}

\begin{document}

\begin{abstract}
  Both Cuckler and Yuster independently conjectured that
  when $n$ is an odd positive multiple of $3$
  every regular tournament on $n$ vertices contains a collection of $n/3$ 
  vertex-disjoint copies of the cyclic triangle.
  Soon after, Keevash \& Sudakov proved that if 
  $G$ is an orientation of a graph on $n$ vertices 
  in which every vertex has both indegree and outdegree at least $(1/2 - o(1))n$, 
  then there exists a collection of vertex-disjoint cyclic triangles
  that covers all but at most $3$ vertices.
  In this paper, we resolve the conjecture of Cuckler and Yuster for sufficiently large $n$.
\end{abstract}

\maketitle

\section{Introduction}

Let $H$ and $G$ be graphs or directed graphs. An \textit{$H$-tiling} of $G$ 
is a collection of vertex-disjoint copies of $H$ in $G$.
An $H$-tiling $\mathcal{C}$ \textit{covers} the set $V(\mathcal{C}) = \bigcup_{C \in \mathcal{C}} V(C)$,
and is called \textit{perfect} or an \textit{$H$-factor} if it covers $V(G)$.

The celebrated Hajnal-Szemer\'edi~Theorem \cite{hajnal1970pcp} states that
for every positive integer $r$, if $n$ is a positive multiple of $r$ and $G$ is a graph on $n$
vertices such that $\delta(G) \ge (1 - 1/r)n$, then $G$ contains a $K_r$-factor.
The case when $r = 3$ is a corollary of an earlier result of Corr\'adi \& Hajnal \cite{corradi1963maximal}.

In this paper, we consider a similar problem in the context of \textit{oriented graphs}, which
are orientations of simple graphs, i.e., oriented graphs are directed graphs in which
there is at most one directed edge between every pair of vertices and no loops.
A \textit{tournament} is an orientation of a complete graph.
For an oriented graph $G$ and $v \in V(G)$,
we denote the 
\textit{out-neighborhood of $v$} and \textit{in-neighborhood of $v$} by $N^+(v)$ and $N^-(v)$, respectively.
We let $N(v) = N^+(v) \cup N^-(v)$ be the \textit{neighborhood of $v$},
and we let $d^+(v) = |N^+(v)|$ and $d^-(v) = |N^-(v)|$ be the \textit{outdegree} and \textit{indegree} of $v$, respectively. 
The \textit{minimum semidegree of $G$} is 
\begin{equation*}
  \delta^0(G) = \min_{v \in V(G)} \{\min\{d^+(v), d^-(v)\}\}.
\end{equation*}
An oriented graph $G$ on $n$ vertices is a \textit{regular tournament} if $d^+(v) = d^-(v) = \frac{n-1}{2}$ for every $v \in G$.

A tournament is \textit{transitive} if it contains no directed cycles, and the unique transitive tournament
on $r$ vertices is denoted $TT_r$.
Up to isomorphism, there are two different tournaments on three vertices: $TT_3$ and the three 
vertex cycle in which the edges
are consistently oriented, which we denote by $C_3$.
We call $C_3$ and $TT_3$ the cyclic and transitive triangles, respectively

There has been some prior work on minimum degree conditions that force
an $H$-factor in directed graphs.
See \cite{wang2000independent} and \cite{CDMT} for work on directed graphs,
and \cite{ottri} and \cite{yuster2003tiling} for oriented graphs.
Also, \cite{treglown2012note} contains many additional interesting embedding problems for oriented graphs.
This paper focuses on the following conjecture that Cuckler and Yuster made independently.
\begin{conjecture}[Cuckler 2008 \cite{cuckler}, Yuster 2007 \cite{yuster2007combinatorial}]
  \label{conj-main}
  If $n$ is an odd positive multiple of $3$, then 
  every regular tournament on $n$ vertices has a cyclic triangle factor.
\end{conjecture}

Keevash \& Sudakov then proved the following approximate version of this conjecture. 
\begin{thm}[Keevash \& Sudakov 2009 \cite{keevash2009triangle}]
  \label{thm:KS}
  There exists $c > 0$ and $n_0$ such
  that for every $n \ge n_0$ the following holds.
  If $G$ is an oriented graph on $n$ vertices
  and $\delta^0(G) \ge (1/2 - c)n$,
  then there exists a cyclic triangle tiling that covers
  all but at most $3$ vertices. 
\end{thm}

A corollary of our main result resolves Conjecture~\ref{conj-main} for large tournaments.
To see that the resolution of Conjecture~\ref{conj-main} is a sharp result, 
consider the following construction from \cite{keevash2009triangle}.
For a positive integer $m$, 
let $G$ be a tournament on $3m$ vertices in which the edges are oriented
so that there exists a partition $\{V_1, V_2, V_3\}$ of $V(G)$ such that 
\begin{itemize}
  \item $|V_1| = m-1$, $|V_2| = m$, and $|V_3| = m+1$;
  \item for $i \in [3]$, the oriented graph induced by $V_i$ has minimum semidegree $\floor{(|V_i|-1)/2}$; and
  \item no edges are directed from $V_2$ to $V_1$, from $V_3$ to $V_2$, and from $V_1$ to $V_3$.
\end{itemize}
We have that,
\begin{equation*}
  \delta^0(G) = 
  \begin{cases}
    (|V_2| - 2)/2 + |V_1| = \frac{n - 4}{2} & \text{if $m$ is even} \\
    (|V_2| - 1)/2 + |V_1| = \frac{n - 3}{2}  & \text{if $m$ is odd}. 
  \end{cases}
\end{equation*}
To see why $G$ has no cyclic triangle factor, 
let $\mathcal{C}$ be a cyclic triangle tiling of $G$ and note that,
for every $C \in \mathcal{C}$, either $C$ has one vertex in each of $V_1$, $V_2$ and $V_3$,
or $V(C) \subseteq V_i$ for some $i \in [3]$.
Therefore,
\begin{equation*}
  |V(\mathcal{C}) \cap V_1| \equiv 
  |V(\mathcal{C}) \cap V_2| \equiv 
  |V(\mathcal{C}) \cap V_3| \pmod 3.
\end{equation*}
Because $|V_1|$, $|V_2|$ and $|V_3|$ are distinct modulo $3$, we have that $V(\mathcal{C}) \neq V(G)$.
Motivated by this example we make the following definitions.
\begin{defi}[Divisibility barrier and $\gamma$-extremal]
  Let $G$ be an oriented graph.
  Call a partition $\mathcal{P}$ of $V(G)$ a \textit{divisibility barrier}
  if either $\mathcal{P}$ is the trivial partition $\{V(G)\}$ and $|V(G)|$ is not divisible by $3$, 
  or if $\mathcal{P}$ has exactly three parts, $V_1$, $V_2$, and $V_3$, such that
  there are no edges directed from $V_2$ to $V_1$, from $V_3$ to $V_2$, and from $V_1$ to $V_3$; 
  and $|V_1|$, $|V_2|$, and $|V_3|$ are not all equivalent modulo $3$.

  For $\gamma > 0$,
  call a partition $\{V_1, V_2, V_3\}$ of $V(G)$ a \textit{$\gamma$-extremal partition} of $G$ if, for every $i\in[3]$, 
  \begin{equation*}
    \left(1/3 -\gamma\right)n\leq|V_1|, |V_2|, |V_3| \leq \left(1/3 +\gamma\right)n, 
  \end{equation*}
  and the number of edges directed from $V_2$ to $V_1$, from $V_3$ to $V_2$, and from $V_1$ to $V_3$ are each at most
  $\gamma n^2$. 
  An oriented graph is called \textit{$\gamma$-extremal} if it contains a $\gamma$-extremal partition.
\end{defi}

The following is the main result of this paper.
\begin{thm}
  \label{thm:main}
  There exists $c > 0$ and $n_0$ such that for every $n \ge n_0$ 
  and for every oriented graph $G$ on $n$ vertices with $\delta^0(G) \ge (1/2 - c)n$
  the following holds.
  $G$ has a cyclic triangle factor if and only if $G$ does not have a divisibility barrier. 
\end{thm}
The following corollary resolves Cuckler and Yuster's conjecture for sufficiently large regular tournaments.
\begin{cor}
  \label{cor-main}
  There exists $n_0$ such that when $n$ is a multiple of $3$ and $n \ge n_0$ the following holds.
  If $G$ is an oriented graph on $n$ vertices and $\delta^0(G) \ge n/2 - 1$, then $G$ has a cyclic triangle factor.
\end{cor}
\begin{proof}
  Assuming that Theorem~\ref{thm:main} holds,
  we only need to show that $G$ does not contain a divisibility barrier.
  For a contradiction, assume that $\{V_1, V_2, V_3\}$ is a divisibility barrier.
  Since $n$ is divisible by $3$ and $|V_1|$, $|V_2|$, and $|V_3|$ are not all equivalent modulo $3$,
  we have that $|V_1|$, $|V_2|$ and $|V_3|$ are distinct modulo $3$.
  Therefore, there exists a labeling $\{i, j, k\} = [3]$ such that $|V_j| \le |V_k| - 2$.
  Suppose that $i + 1 \equiv j \pmod 3$.
  There exists $v \in V_i$ such that $|N^+(v) \cap V_i| < |V_i|/2$, so 
  \begin{equation*}
    d^+(v) = |N^+(v) \cap V_i| + |N^+(v) \cap V_j| + |N^+(v) \cap V_k| < \frac{|V_i|}{2} + |V_j| \le \frac{|V_i| + |V_j| + |V_k| - 2}{2} = \frac{n}{2} - 1,
  \end{equation*}
  a contradiction. A similar argument holds when $i - 1 \equiv j \pmod 3$.
\end{proof}

We are not sure how large the constant $c$ can be in Theorem~\ref{thm:main},
and we do not compute the value of $c$ that our proof implies.
An example of Keevash \& Sudakov \cite{keevash2009triangle}, which we will present below,
implies that the constant $c$ cannot be larger than $1/18$.
This suggests the following problem.
\begin{problem}\label{problem1}
  What is the smallest $\phi \ge 0$ such that there exists $n_0$ such that for every $n \ge n_0$ 
  every oriented graph $G$ on $n$ vertices with $\delta^0(G) \ge (4/9 + \phi)n$
  contains either a divisibility barrier or a cyclic triangle factor? 
\end{problem}
For $m \ge 1$, let $G$ be an oriented graph on $n = 9(m + 1)$ vertices and
let $\mathcal{P} = \{V_1, V_2, V_3\}$ be a partition of $V(G)$ 
such that $|V_1| = 3m + 1$ and $|V_2| = |V_3| = 3m + 4$. 
Suppose that for every pair $i,j \in [3]$ such that $j \equiv i + 1 \pmod 3$,
there is a directed edge from every vertex in $V_i$ to
every vertex in $V_j$.
Further suppose that, for every $i \in [3]$, 
the vertices of $V_i$ can be cyclically
ordered so that, for every $v \in V_i$, the intersection of the out-neighborhood of $v$ and $V_i$
is exactly the $(|V_i|-1)/3$ vertices that succeed $v$ in this ordering.
Note that every cyclic triangle has at least one vertex in $V_1$, 
so $|\mathcal{C}| \le |V_1| = n/3 - 2$ for every cyclic triangle tiling $\mathcal{C}$,
and 
\begin{equation*}
  \delta^0(G) = \frac{|V_2| - 1}{3} + |V_1| = \frac{4n}{9} - 2.
\end{equation*}

Note that $\mathcal{P}$ is not a divisibility barrier, because the
three parts of $\mathcal{P}$ all have the same size modulo $3$.
Additionally, $G$ cannot have a divisibility barrier, because for every 
partition $\mathcal{P}' \neq \mathcal{P}$, such that $|\mathcal{P'}| \le 3$, 
there exists a part $U \in \mathcal{P}'$ such that
there exist $x, y \in U$ such that $x$ and $y$ are in different parts, say $V_i$ and $V_j$, of $\mathcal{P}$.
But then either $\mathcal{P}'$ is not a divisibility barrier or 
$U$ contains all of the vertices of the third part, say $V_k$, of $\mathcal{P}$.
Then, by similar logic, 
one can argue that either $\mathcal{P}'$ is not a divisibility barrier or $U = V(G)$.
Since $3$ divides $|V(G)|$, $\mathcal{P}'$ is not a divisibility barrier.

Note that if the famous Caccetta-H\"aggkvist Conjecture \cite{caccetta1978minimal} is false and, for some $\psi > 1/3$ and all $n$, there exists a $C_3$-free oriented graph on $n$ vertices with minimum semidegree $\psi n$,
then a similar example would imply that $\phi$ must be strictly 
greater than $0$ in Problem~\ref{problem1}.

\subsection{Additional Definitions and Notation}
Let $G$ be an oriented graph, $u,v \in V(G)$, and $A,B \subseteq V(G)$.
We let $N^+(u)$ be the out-neighborhood 
of $u$,
$N^+(u, B) = N^+(u) \cap B$, $d^+(u, B) = |N^+(u, B)|$,
and $e^+(A, B) = \sum_{u \in A} d^+(u, B)$.
We define $N^-(u)$, $N^-(u, B)$, $d^-(u, B)$ and $e^-(A, B)$ similarly.
If $uv \in E(G)$, we let $d^{+,-}(uv, A) = |N^+(u) \cap N^-(v) \cap A|$
and $d^{-,+}(uv, A)$, $d^{+,+}(uv, A)$, and $d^{-,-}(uv, A)$ are all defined similarly.
We also let $d^{\sigma, \tau}(uv) = d^{\sigma, \tau}(uv, V(G))$, for $\sigma, \tau \in \{-, +\}$.
We let $E(A)$ be the set of edges in the oriented graph induced by $A$,
and let $e(A) = |E(A)|$.
We let $\overline{A} = V(G) \setminus A$.

We define $E^+(A, B) = \{uv \in E(G) : u \in A \text{ and } v \in B\}$
and $E^-(A, B) = E^+(B, A)$.
We will often write cyclic and transitive triangles $C$ as $abc$ when
$V(C) = \{a, b, c\}$.
For $V_1, V_2, V_3 \subseteq V(G)$, 
$\cyc(V_1, V_2, V_3)$ and $\trn(V_1, V_2, V_3)$ count, respectively, 
the number of cyclic and transitive triangles with vertex set
$\{v_a, v_b, v_c\}$ such that
$\{a, b, c\} = [3]$ and $v_i \in V_i$ for $i \in [3]$.
We abbreviate $\cyc(A, A, A)$ and $\trn(A, A, A)$ as 
$\cyc(A)$ and $\trn(A)$, respectively.
We will often replace $\{v\}$ with $v$ in this notation.

We define the \textit{strong $\beta$-out-neighborhood} of $A$ to be the set of 
vertices $x \in \overline{A}$ such that $d^{-}(x, A) \ge |A| - \beta n$
and we denote this set by $\strong{A}{+}{\beta}$.
We define the \textit{strong $\beta$-in-neighborhood} similarly.

Throughout the paper, we write $0<\alpha \ll \beta \ll \gamma$ 
to mean that we can choose the constants
$\alpha, \beta, \gamma$ from right to left. More
precisely, there are increasing functions $f$ and $g$ such that, given
$\gamma$, whenever we choose $\beta \leq f(\gamma)$ and $\alpha \leq g(\beta)$, all
calculations needed in our proof are valid. 
Hierarchies of other lengths are defined in the obvious way.
For real numbers $x$ and $y$,
we write $x = y \pm c$ to mean that $y - c \le x \le y+c$.

In our proof, we use a very small part of the theory and notation developed in 
\cite{keevash2015geometric}, \cite{keevash2015polynomial},\cite{han2017decision},\cite{lo2013minimum},
and \cite{lo2015f}. It is based on the absorbing method of 
R\"odl, Ruci\'nksi and Szemer\'edi \cite{rodl2006dirac}.
This theory was developed for hypergraphs, so 
we define, for every oriented graph $G$, the hypergraph $H(G)$ 
to be the $3$-uniform hypergraph in which $xyz$ is in an edge 
if and only if $xyz$ is a cyclic triangle in $G$.
Clearly a cyclic triangle factor in $G$ is equivalent to a
perfect matching in $H(G)$.

Let $V$ be a set of order $n$ and
let $\mathcal{P} = \{V_1, \dotsc, V_d\}$ be a partition of $V$.
We say that $\mathcal{P}$ is \textit{trivial} if $|\mathcal{P}| = 1$,
and, for $\eta > 0$,  
we call $\mathcal{P}$ an \textit{$\eta$-partition} if $|V_i| \ge \eta n$ for every $i \in [d]$.
Let $H$ be a $k$-uniform hypergraph with vertex set $V$.
We let
\begin{equation*}
  \delta_1(H) =  \min_{v \in V} \{ | \{e \in E(H) : v \in e \} | \}.
\end{equation*}
For every subset $U$ of $V$ the \textit{index vector with respect to $\mathcal{P}$}, denoted
$\mathbf{i}_{\mathcal{P}}(U)$, is the vector defined by 
\begin{equation*}
  \ivec(U) = \left(|U \cap V_1|, |U \cap V_2|, \dotsc, |U \cap V_d|\right).
\end{equation*}
Let $\Ivec(H) = \{\ivec(e) : e \in E(H)\}$ be the set of \textit{edge-vectors} and,
for $\mu > 0$,  let 
\begin{equation*}
  \uIvec(H) = \{ \mathbf{v} \in \Ivec(H) : \text{ there are at least $\mu n^k$ edges $e$ in $H$ such that $\mathbf{v} = \ivec(e)$} \},
\end{equation*}
be the set of \textit{$\mu$-robust edge-vectors}.

We call an additive subgroup of $\mathbb{Z}^d$ a \textit{lattice},
and we let $\lattice(H)$ and $\ulattice(H)$ be the lattices generated by $\Ivec(H)$ and $\uIvec(H)$, respectively.
Clearly if $M$ is a collection of vertex-disjoint edges in $H$, then $\ivec(V(M)) \in \lattice(H)$.
Therefore, $H$ does not have a perfect matching if $\ivec(V(H)) \notin \lattice(H)$.
For example, suppose that $G$ is an oriented graph with a divisibility barrier.
If $|V(G)|$ is not divisibility by $3$, then $\ivec(V(G)) \notin \lattice(H(G))$ when $\mathcal{P} = \{ V(G) \}$.
Otherwise, there exists a partition $\mathcal{P} = \{V_1, V_2, V_3\}$ of $V(G)$ such that
the entries of $\ivec(V(G))$ are each different modulo $3$, and, for every $e \in E(H(G))$,
the entries of $\ivec(e)$ are each the same modulo $3$, so $\ivec(V(G)) \notin \lattice(H(G))$.
We let $\uvec_i \in \mathbb{Z}^d$ be the \textit{$i$th unit vector}, i.e., $\uvec_i$ is the vector
in which the $i$th component is $1$ and every other component is $0$.
A \textit{transferral} is a vector $\mathbf{v}$ in $\mathbb{Z}^d$ such that
$\mathbf{v} = \uvec_i - \uvec_j$ for distinct $i$ and $j$ in $[d]$.
A \textit{$2$-transferral $\mathbf{v} \in \ulattice(H)$} is a  
transferral such that  $\mathbf{v} = \mathbf{v}_1 - \mathbf{v}_2$ for $\mathbf{v}_1, \mathbf{v}_2 \in \uIvec(H)$.
We say that $\ulattice(H)$ is \textit{$2$-transferral-free} if $\ulattice(H)$ does not contain a $2$-transferral.

Let $x$ and $y$ be vertices in $V$ and let $\beta > 0$.
A set $S \subseteq V$ is called an \textit{$(H,x,y)$-linking $(k\ell-1)$-set} if both $H[S\cup \{x\}]$ and $H[S\cup \{y\}]$ have perfect matchings
and $|S| = k \ell - 1$.
The vertices $x$ and $y$ are \textit{$(H, \beta, \ell)$-reachable} if there are at least $\beta n^{{\ell}k-1}$ $(H, x,y)$-linking
$(k \ell - 1)$-sets.
A set $U \subseteq V$ is \textit{$(H, \beta, {\ell})$-closed} if every pair of 
distinct vertices $x$ and $y$ in $U$ is $(H, \beta, {\ell})$-reachable. 
Call a partition $\mathcal{P} = \{V_1, \dotsc, V_d\}$ of $V(H)$ a \textit{$(H, \beta, {\ell})$-closed partition} if,
for every $i \in [d]$, $V_i$ is $(H, \beta, {\ell})$-closed.
For every vertex $x \in V(H)$, let $\tilde{N}_H(\beta,{\ell},x)$ be the set of vertices $y$ such that $x$ and $y$ 
are $(H, \beta, {\ell})$-reachable.

\section{Proof of Theorem~\ref{thm:main}}

\subsection{Overview}

Our proof follows the stability method, i.e., we divide the proof into two cases depending on whether  $G$ is $\gamma$-extremal.
The case when $G$ is $\gamma$-extremal is handled in the following lemma, the proof of which
we defer until Section~\ref{sec:extremal}.
\begin{lemma}[Extremal case]\label{lem:extremal}
  Suppose that $0 < 1/n \ll c, \gamma \ll 1$, $n$ is divisible by $3$, and $G$ is an oriented graph on $n$ vertices.
  If $\delta^0(G) \ge (1/2 - c)n$ and $G$ is $\gamma$-extremal, then $G$ has a triangle factor if and only if $G$ does not have a divisibility barrier.
\end{lemma}

The proof in the case when $G$ is not $\gamma$-extremal follows the absorbing method, and 
the following lemma of Lo \& Markstr{\"o}m \cite{lo2015f} serves as our absorbing lemma.
\begin{lemma}[Absorbing Lemma - Lemma~1.1 in \cite{lo2015f}]
  \label{lem:absorbing}
  Given $k$ and 
  \begin{equation*}
    0 < 1/n \ll \alpha \ll \eta \ll \beta \ll 1/{\ell} \le 1
  \end{equation*}
  the following holds.
  If $H$ is a $k$-uniform hypergraph on $n$ vertices and $V(H)$ is $(H, \beta, {\ell})$-closed, then 
  there exists a subset $U$ of $V(H)$ where $|U|$ is at most $\eta n$ and divisible by $k$ 
  such that for every $W \subseteq V(H) \setminus U$ with $|W|$ at most
  $\alpha n$ and divisible by $k$ there exists a perfect matching of $H[U \cup W]$.
\end{lemma}

Lemma~\ref{lem:absorbing} and Theorem~\ref{thm:KS} imply the following lemma. 
\begin{lemma}\label{lem:non-extremal}
  Suppose that $0 < 1/n \ll c \ll \beta \ll 1/{\ell} < 1$, $n$ is divisible by $3$,
  and that $G$ is an oriented graph on $n$ vertices.
  If $\delta^0(G) \ge (1/2 - c)n$ and $V(G)$ is $(H(G), \beta, {\ell})$-closed, 
  then $G$ has a triangle factor.
\end{lemma}
\begin{proof}
Introduce constants $\eta$ and $\alpha$ so that 
\begin{equation*}
  1/n \ll \alpha \ll \eta \ll c \ll \beta.
\end{equation*}
Let $U \subseteq V(G)$ be the set guaranteed by Lemma~\ref{lem:absorbing}
and let $G' = G - U$. Note that 
\begin{equation*}
  \delta^0(G') \ge \delta^0(G) - |U| \ge (1/2 - c - \eta)n \ge (1/2 - 2 c)|G'|.
\end{equation*}
Therefore, by Theorem~\ref{thm:KS}, with $2c$ and $G'$ playing the roles of $c$ and $G$, respectively, 
there exists a cyclic triangle tiling $\mathcal{C}$ of $G'$ such that
if $W = V(G') \setminus V(\mathcal{C})$, then $|W| \in \{0, 3\}$.
Since $|W| \le \alpha n$ and $|W|$ is divisible by $3$, there
exists $\mathcal{C'}$ a triangle factor of $G[U \cup W]$.
Hence, $\mathcal{C} \cup \mathcal{C}'$ is a triangle factor of $G$.
\end{proof}

Note that, with Lemmas~\ref{lem:extremal} and \ref{lem:non-extremal}, the following lemma implies Theorem~\ref{thm:main}.
\begin{lemma}\label{lem:not-closed}
  Suppose that $0 < 1/n \ll c \ll \beta \ll \gamma < 1$ and that $G$ is an oriented graph on $n$ vertices.
  If $\delta^0(G) \ge (1/2 - c)n$, and, for every positive integer $\ell \le 1000$, 
  we have that $V(G)$ is not $(H(G), \beta, {\ell})$-closed, then $G$ is $\gamma$-extremal.
\end{lemma}

The proof of Lemma~\ref{lem:not-closed} relies on the following four lemmas.

\begin{lemma}
  \label{lem:merge}
  Suppose that $0 < 1/n \ll \beta \ll \beta', \mu \ll 1/d, 1/{\ell} \le 1$,
  $G$ is an oriented graph on $n$ vertices, and 
  that $\mathcal{P} = \{V_1, \dotsc, V_d\}$ is a partition of $V(G)$.
  If $\mathcal{P}$ is $(H(G), \beta', {\ell})$-closed and there exists a
  $2$-transferral $\uvec_i - \uvec_j \in \ulattice(H(G))$, then 
  the partition formed by merging $V_i$ and $V_j$ is 
  $(H(G), \beta, 4\ell + 1)$-closed.
\end{lemma}

\begin{lemma}
  \label{lem:partition-existence}
  Suppose that $0 < 1/n \ll c \ll \beta \ll 1$ and that $G$ is an oriented graph on $n$ vertices
  such that $\delta^0(G) \ge (1/2 - c)n$.
  Then there exists $\mathcal{P} = \{V_1, \dotsc, V_d\}$ a $(H(G), \beta, 8)$-closed partition of $V(G)$
  such that $d \le 4$ and $|V_i| > n / 9$ for every $i \in [d]$.
\end{lemma}

\begin{lemma}
  \label{lem:not-closed-1}
  Suppose that $0 < 1/n \ll c \ll \mu \ll \alpha \ll \eta < 1$  and $G$ is an oriented graph on $n$ vertices
  such that $\delta^0(G) \ge (1/2 - c)n$.
  If $\mathcal{P}$ is a non-trivial $\eta$-partition of $V(G)$ and
  $\ulattice(H(G))$ is $2$-transferral-free, 
  then there exists $A \in \mathcal{P}$ such that $\cyc(A, A, \overline{A}) \le \alpha n^3$.
\end{lemma}

\begin{lemma}
  \label{lem:not-closed-2}
  Suppose that $0 < 1/n \ll c \ll \mu, \alpha \ll \gamma, \eta < 1$ and $G$ is an oriented graph on $n$ vertices
  such that $\delta^0(G) \ge (1/2 - c)n$.
  If $\mathcal{P}$ is a non-trivial $\eta$-partition of $V(G)$,
  $\ulattice(H(G))$ is $2$-transferral-free, 
  and there exists $A \in \mathcal{P}$ such that $\cyc(A, A, \overline{A}) \le \alpha n^3$, 
  then $G$ is $\gamma$-extremal.
\end{lemma}
We prove Lemmas~\ref{lem:merge} and \ref{lem:partition-existence} in Section~\ref{sec:lemmas},
Lemma~\ref{lem:not-closed-1} in Section~\ref{sec:not-closed-1}, and 
Lemma~\ref{lem:not-closed-2} in Section~\ref{sec:not-closed-2}.

\begin{proof}[Proof of Lemma~\ref{lem:not-closed}]
  Introduce additional constants, $\beta'$, $\mu$ and $\alpha$ so that 
  \begin{equation*}
    0 < \beta \ll \beta' \ll \mu \ll \alpha  \ll \gamma \ll 0.1.
  \end{equation*}
  By Lemma~\ref{lem:partition-existence},  with $\beta'$ and $\ell'$ playing the roles of $\beta$ and $\ell$, respectively,
  there exists $\mathcal{P}' = \{V'_1, \dotsc, V'_{d'}\}$ a $(H(G), \beta', 8)$-closed $0.1$-partition of $V(G)$ such that $d' \le 4$.
  If $\uplattice(H(G))$ contains a $2$-transferral $\uvec_i - \uvec_j$ for distinct $i$ and $j$ in $[d']$,
  then \textit{merge the parts that correspond to the $2$-transferral}, i.e.,
  consider the new partition $\mathcal{P}' - V'_i - V'_j + (V'_i \cup V'_j)$.
  Continue to merge the parts that correspond to $2$-transferrals
  until we have a partition $\mathcal{P}$ such that $\ulattice(H(G))$ is $2$-transferral-free.
  By Lemma~\ref{lem:merge}, we can assume that 
  $\mathcal{P}$ is an $(H(G), \beta, \ell)$-closed $0.1$-partition of $V(G)$
  for some $\ell \le 5^3 \cdot 8 = 1000$.
  If $|\mathcal{P}| = 1$, then $V(G)$ is $(H(G), \beta, \ell)$-closed which contradicts our assumptions,
  so we can assume that $\mathcal{P}$ is non-trivial.
  By Lemma~\ref{lem:not-closed-1}, there exists $A \in \mathcal{P}$ such that
  $\cyc(A, A, \overline{A}) \le \alpha n^3$ and by 
  Lemma~\ref{lem:not-closed-2} we have that $G$ is $\gamma$-extremal. 
\end{proof}

\subsection{Proofs of Lemmas~\ref{lem:merge} and \ref{lem:partition-existence}}\label{sec:lemmas}

We start this section with a proof of Lemma~\ref{lem:merge}.
\begin{proof}[Proof of Lemma~\ref{lem:merge}]
  Let $\xi$ be such that $\beta \ll \xi \ll \beta' \ll 1/\ell$.
  Note that distinct vertices $u_0$ and $v_0$ are $(H(G), \beta, 4\ell + 1)$-reachable if there exist
  at least $\xi n^{12 \ell + 2}$ ordered $(12\ell+2)$-tuples $\mathcal{T}$ that each
  can be permuted to form $(12\ell + 2)$-tuples
  $(u_1, \dots, u_{12\ell + 2})$ 
  and 
  $(v_1, \dots, v_{12\ell + 2})$
  such that $u_{3j}u_{3j+1}u_{3j+2}$ and $v_{3j}v_{3j+1}v_{3j+2}$ are
  both cyclic triangles for every $0 \le j \le 4\ell$.
  This is because $\beta < \xi^2/2$; and
  there are only at most $(12\ell+2)! < 1/\xi$ possible orderings for every such tuple; and
  only at most 
  \begin{equation*}
    2 (12 \ell + 2) n^{12\ell+1} + n \cdot (12{\ell} + 2)^2 \cdot n^{12{\ell}} < \xi n^{12{\ell} + 2}/2
  \end{equation*}
  such $(12\ell+2)$-tuples that contain $u_0$ or $v_0$ or that have repeated vertices.

  We will first show that if $u, v \in V(G)$ are $(H(G), \beta', \ell)$-reachable, then
  $u$ and $v$ are also $(H(G), \beta, 4\ell + 1)$-reachable.
  In particular, this will imply that, because $\mathcal{P}$ is $(H(G), \beta', \ell)$-closed,
  $\mathcal{P}$ is also $(H(G), \beta, 4\ell + 1)$-closed.
  Let $\mathcal{T}$ be the set of $(12 \ell + 2)$-tuples $(x_1, x_2, \dotsc, x_{12\ell+2})$
  such that 
  \begin{itemize}[noitemsep]
    \item $\{x_1, x_2, \dotsc, x_{3\ell - 1}\}$ is a $(H(G),u,v)$-linking $(3\ell -1)$-set, and
    \item $G[\{x_{3j},x_{3j+1},x_{3j+2}\}]$ is a cyclic triangle for $\ell \le j \le 4\ell$.
  \end{itemize}
  Since $u$ and $v$ are $(H(G), \beta', \ell)$-reachable, there are at least
  $\beta' n^{3\ell - 1}$ ways to select the first $3\ell - 1$ entries of a tuple in $\mathcal{T}$.
  For every such selection, there are exactly $(6\cyc(V(G)))^{3\ell + 1}$ ways to select the remaining
  $9\ell + 3$ entries of a tuple in $\mathcal{T}$.
  To see that $6 \cyc(V(G)) \ge 2(\beta')^2 n^3$,
  let $V_q$ be the part in $\mathcal{P}$ of largest cardinality.
  We have that $|V_q| \ge n/d \ge \beta' n$.
  Because $V_q$ is $(H(G), \beta', \ell)$-closed, for every
  $v \in V_q$, the set $\tilde{N}_{H(G)}(\beta', \ell, v)$ is not empty, which implies that
  $v$ is in at least $\beta' n^2$ cyclic triangles. 
  Therefore, because $6 \cyc(V(G))$ is 
  the number of ordered triples $(x,y,z)$ such that $G[\{x,y,z\}]$ is a cyclic triangle in $G$,
  we have that $6 \cyc(V(G)) \ge 2 (\beta')^2 n^3$. Hence,
  \begin{equation*}
    |\mathcal{T}| \ge \beta' n^{3\ell - 1} \cdot \left( 2(\beta')^2n^3 \right)^{3\ell + 1} > \xi n^{12\ell + 2},
  \end{equation*}
  so $u$ and $v$ are $(H(G), \beta, 4\ell+1)$-reachable.

  Now we will complete the proof by showing that if $u_0 \in V_i$ and $v_3 \in V_j$,
  then $u_0$ and $v_3$ are $(H(G), \beta, 4\ell + 1)$-reachable.
  By assumption, there are $A, B \in \mathcal{P}$ such that
  $\cyc(V_i, A, B)$ and $\cyc(A, B, V_j)$ are both at least $\mu n^3$.
  Let $\mathcal{T}$ be the set of $(12 \ell + 2)$-tuples 
  \begin{equation*}
    (v_0, v_1, v_2, u_1, u_2, u_3, w_1, \dotsc, w_{12\ell-4})
  \end{equation*}
  that satisfy the following:
  \begin{itemize}[noitemsep]
    \item $v_0, v_1, v_2$ is a cyclic triangle 
      with $v_0 \in V_i$, $v_1 \in A$ and $v_2 \in B$;
    \item $u_1, u_2, u_3$ is a cyclic triangle 
      with $u_1\in A$ and $u_2 \in B$, and $u_3 \in V_j$; and
    \item
      \{$w_{i(3{\ell} - 1) + 1}, \dotsc, w_{(i+1)(3{\ell} - 1)}\}$
      is a $(H(G), u_i, v_i)$-linking $(3\ell - 1)$-set for $i \in \{0, 1, 2, 3\}$.
  \end{itemize}
  Since $\cyc(V_i, A, B)$ and $\cyc(A, B, V_j)$ are both at least $\mu n^3$ and
  $V_i$, $V_j$, $A$, $B$  are all $(H(G), \beta', {\ell})$-closed, we have that
  \begin{equation*}
    |\mathcal{T}| \ge (\mu n^3) \cdot (\mu n^3) \cdot (\beta' n^{3{\ell} - 1})^4
    \ge \xi n^{12{\ell} + 2},
  \end{equation*}
  so $u_0$ and $v_3$ are $(H(G), \beta, 4{\ell} + 1)$-reachable.
\end{proof}

To prove Lemma~\ref{lem:partition-existence}, we use the following lemma of Han~\cite{han2017decision}.
\begin{lemma}[Lemma~3.8 in \cite{han2017decision}]
  \label{lem:partition}
  Given $0 < \alpha \ll \delta, \delta'$, there exists a constant $\beta > 0$ for which
  the following holds.
  Let $H$ be a $k$-uniform hypergraph on $n$ vertices where $n$ is sufficiently large.
  Assume that $\left|\tilde{N}_H(\beta, 1, v)\right| \ge \delta' n$ for every $v \in V(H)$
  and $\delta_1(H) \ge \delta \binom{n-1}{k-1}$.
  Then for $d \le \min\{\floor{1/\delta}, \floor{1/\delta'}\}$ there exists 
  an $(H, \beta, 2^{\floor{1/\delta} - 1})$-closed partition
  $\mathcal{P} = \{V_1, \dotsc, V_d\}$ of $V(H)$
  such that for every $i \in \{1, \dotsc, d\}$ we have that $|V_i| \ge (\delta' - \alpha)n$.
\end{lemma}
To apply Lemma~\ref{lem:partition} in our context, we need lower bounds on $\delta_1(H(G))$ and 
$\left|\tilde{N}_{H(G)}(\alpha, 1, v)\right|$ for every vertex $v \in V(G)$ when $G$ is an oriented graph with
sufficiently high minimum semidegree. The following series of lemmas provide these lower bounds.
Note that Lemmas~\ref{in=out}, \ref{min_vertex_degree}, and \ref{p0} are also used in other sections.
\begin{lemma}\label{in=out}
  Suppose that $c > 0$ and $G = (V, E)$ is an oriented graph on $n$ vertices such that $\delta^0(G) \ge (1/2 - c)n$. 
  Then, for every $A \subseteq V$,
  \begin{equation*}
    e^+(A, \overline{A}), e^-(A, \overline{A}) = \frac{ |A| \cdot |\overline{A}|}{2} \pm c|A|n.
  \end{equation*}
\end{lemma}
\begin{proof}
  We first get a lower bound for $e^+(A, \overline{A})$ as follows,
  \begin{equation*}
    e^+(A, \overline{A}) = 
    \sum_{x \in A}d^+(x) - d^+_A(x) \ge  
    \left(\sum_{x \in A}d^+(x)\right) - \binom{|A|}{2} \ge 
    |A|\left(\frac{1}{2} - c\right)n - \frac{|A|^2}{2} =
    \frac{|A|\cdot|\overline{A}|}{2} - c|A|n.
  \end{equation*}
  By a similar computation, we have that $e^-(A, \overline{A}) \ge |A||\overline{A}|/2 - c|A|n$.
  The fact that 
  \begin{equation*}
    e^+(A, \overline{A}) + e^-(A, \overline{A}) \le |A| \cdot |\overline{A}|,
  \end{equation*}
  then implies the upper bounds.
\end{proof}

The following lemma appears in \cite{keevash2009triangle}.
We provide a proof for completeness.
\begin{lemma}
  \label{min_vertex_degree}
  Suppose that $c > 0$ and $G = (V, E)$ is an oriented graph on $n$ vertices such that $\delta^0(G) \ge (1/2 - c)n$. 
  Then, 
  $\delta_1(H(G)) = n^2/8 \pm 2cn^2$. 
\end{lemma}
\begin{proof}
  Let $u \in V$.  We will show that $\cyc(u, V, V) = n^2/8 \pm 2cn^2$.
  Let $m = d^+(u)$ and assume $d^+(u) \le d^-(u)$, so $m \le (n-1)/2$.
  By the minimum semidegree condition, $|\overline{N(u)}| \le 2cn$, so
  \begin{equation*}
    e^+(N^+(u), \overline{N(u)}) \le m \cdot 2cn.
  \end{equation*}
  Since we can assume $c < 1/2$, 
  \begin{equation*}
    n^2/4 - cn^2 < n^2/4 - c^2 n^2 \le m(n-m) \le n^2/4. 
  \end{equation*}
  With Lemma~\ref{in=out} and the fact that $m \le n/2$, we then have that 
  \begin{equation*}
    e^+(N^+(u), N^-(u)) = e^+(N^+(u), \overline{N^+(u)}) - e^+(N^+(u), \overline{N(u)})
    = \frac{m(n-m)}{2} \pm 3cnm
    = n^2/8 \pm 2cn^2. 
  \end{equation*}
  Applying a similar argument when $d^-(u) \le d^+(u)$ proves the lemma.
\end{proof}

\begin{lemma}\label{p0}
  Suppose that $c > 0$ and $G = (V, E)$ is an oriented graph on $n$ vertices such that $\delta^0(G) \ge (1/2 - c)n$. 
  Then, for every edge $uv \in E$,
  \begin{equation*}
    d^{-,+}(uv) - d^{+,-}(uv) = \pm 4cn.
  \end{equation*}
  This further implies that, for every pair of disjoint subsets $A$ and $B$ of $V$,
  \begin{equation*}
    \sum_{uv \in E^+(A, B)} \cyc(u, v, V) \ge \frac{e^{+}(A, B)^2}{2|A|} - cn^3.   
  \end{equation*}
\end{lemma}
\begin{proof}
  To prove the first part of the lemma, note that 
  \begin{equation*}
      d^{+,-}(xy)  = d^+(x) - d^{+,+}(xy) - |N^+(x) \setminus N(y)|, 
  \end{equation*}
  and similarly $d^{-,+}(xy) = d^+(y) - d^{+,+}(xy) - |N^+(y) \setminus N(x)|$.
  Therefore, 
  \begin{equation}
    d^{-,+}(xy) - d^{+,-}(xy) = (d^+(x) - d^+(y)) +  (|N^+(y) \setminus N(x)| - |N^+(x) \setminus N(y)|) = \pm 4cn,
  \end{equation}
  by the minimum semidegree condition.

  To prove the second part of the lemma, note that if $u \in A$ and $v \in N^+(u, B)$, then
  \begin{equation*}
    d^{+,-}(uv) \ge d^-(v, N^+(u, B)).
  \end{equation*}
  This implies that 
  \begin{equation*}
    \sum_{v \in N^+(u, B)} d^{+,-}(uv) \ge  \sum_{v \in N^+(u, B)} d^-(v, N^+(u, B))
    = e(N^+(u, B)) \ge d^+(u, B)\left(\frac{d^+(u, B) - 2cn}{2}\right),
  \end{equation*}
  where the last inequality follows from the fact that $n - 2\delta^0(G) \le 2cn$.
  This observation, with the first part of the lemma gives us that 
  \begin{equation*}
    \sum_{v \in N^+(u, B)} \cyc(u,v,V) = 
    \sum_{v \in N^+(u, B)} d^{-,+}(uv) \ge
    \sum_{v \in N^+(u, B)} \left( d^{+,-}(uv) - 4cn \right) 
    \ge d^+(u, B)\left(\frac{d^+(u, B)}{2} - 5cn\right).
  \end{equation*}
  Letting $m = e^+(A, B)$ and $|A| = a$, we have that, by the convexity of $f(x) = x^2$, 
  \begin{equation*}
    \sum_{uv \in E^+(A, B)} \cyc(u,v,V) \ge 
    \sum_{u \in A} d^+(u, B)\left(\frac{d^+(u, B)}{2} - 5cn\right) \ge
    \frac{a}{2}\left(\frac{m}{a}\right)^2 - m \cdot 5cn \ge \frac{m^2}{2a} - cn^3,
  \end{equation*}
  where the last inequality follows because, by Lemma~\ref{in=out}, $m \le n^2/5$.
\end{proof}

\begin{lemma}
  \label{l1}
  Suppose $0 < 1/n \ll c \ll \alpha \ll 1$ 
  and that $G = (V, E)$ is an oriented graph on $n$ vertices such that $\delta^0(G) \ge (1/2 - c)n$.
  Then, for every $v \in V$, 
  \begin{equation*}
    \left|\tilde{N}_{H(G)}(\alpha, 1, v)\right| \ge \left(\frac{1}{8} - 10\alpha\right)n.
  \end{equation*}
\end{lemma}
\begin{proof}
  Fix $v \in V$ and let $N = \tilde{N}_{H(G)}(\alpha, 1, v)$ and $\overline{N} = V \setminus N$.
  For every $U \subseteq V$, let  $T(U)$ be the set of ordered triples
  $(x,y,u)$ such that $u \in U$, $xy \in E$ and both $vxy$ and $xyu$ are cyclic triangles.
  By definition, $u \in \overline{N}$ if and only if $|T(\{u\})| < \alpha n^2$, so 
  \begin{equation*}\label{TNsize}
    |T(N)| = |T(V)| - |T(\overline{N})| \ge |T(V)| - \alpha n^3.
  \end{equation*}
  For every $u \in V$, Lemma~\ref{min_vertex_degree} implies that $|T(\{u\})| \le (1/8 + 2c) n^2$, so 
  \begin{equation*}
    (1/8 + 2c) n^2 \cdot |N| \ge |T(N)| \ge |T(V)| - \alpha n^3
  \end{equation*}
  Therefore, to show that $|N| \ge (1/8 - 10\alpha)n$ and complete the proof, it suffices to 
  prove that $|T(V)| \ge (1/64 - 2c)n^3$.

  Let $m = e^+(N^+(v), N^-(v))$.
  By Lemma~\ref{p0}, we have that
  \begin{equation}\label{TVsize}
    |T(V)| = \sum_{xy \in E^+(N^+(v), N^-(v))} \cyc(x,y,V) \ge \frac{m^2}{2 d^+(v)} - cn^3.
  \end{equation}
  This completes the proof, because Lemma~\ref{min_vertex_degree} implies that $m \ge (1/8 - 2c)n^2$, so
  \begin{equation*}
    \frac{m^2}{2d^+(v)} \ge \frac{mn}{2} \cdot \frac{1/8 - 2c}{1/2 + c} 
    \ge (1/16 - c)n^3 \cdot (1/4 - 5c) \ge (1/64 - c)n^3. \qedhere
  \end{equation*}
\end{proof}

\begin{proof}[Proof of Lemma~\ref{lem:partition-existence}]
  Let $\alpha = 1/1000$, $\delta = 2/9$ and $\delta' = 1/8 - 1/100$.
  By Lemmas~\ref{min_vertex_degree} and \ref{l1}, we have that $\delta_1(H(G)) \ge (1/8 - 2c)n^2 > \delta \binom{n-1}{2}$ and 
  that $\left|\tilde{N}_{H(G)}(\alpha, 1, v)\right| \ge (1/8 - 10 \alpha)n = \delta' n$ for every $v \in V(G)$.
  Therefore, by Lemma~\ref{lem:partition}, there exists $\mathcal{P} = \{V_1, \dotsc, V_d\}$
  a $(H(G), \beta, 8)$-closed partition of $V(G)$ such that 
  such that $d \le \min \{ \floor{1/\delta}, \floor{1/\delta'} \} = 4$
  and $|V_i| \ge (\delta' - \alpha)n > n/9$ for every $i \in [d]$.
\end{proof}

\subsection{Proof of Lemma~\ref{lem:not-closed-1}}\label{sec:not-closed-1}

In this section, we prove Lemma~\ref{lem:not-closed-1}.
In an effort to explain the structure of the proof at a high level, 
we first informally discuss how to derive a contradiction in 
the following situation.
Suppose that $G$ is a regular tournament on $n$ vertices
that has an $\eta$-partition $\mathcal{P}$ of $V(G)$ with
three distinct parts $A, B, D \in \mathcal{P}$ such that 
\begin{itemize}
  \item $A$ is the largest part in $\mathcal{P}$,
  \item $G[A], G[B]$ are transitive tournaments,
  \item every cyclic triangle with two vertices in $A$ has
    one vertex in $B$, and 
  \item every cyclic triangle with two vertices in $B$ has one vertex in $D$.
\end{itemize}
These conditions are
an idealized version of the conditions we will
use to produce a contradiction to prove Lemma~\ref{lem:not-closed-1}.

Let $C = \overline{A \cup B}$.
Since $G[A]$ is a transitive tournament,
there exists $x^+,x^- \in A$ such that 
$d^+(x^+, A) = d^-(x^-, A) = |A| - 1$.
This implies that when $C^- = N^-(x^+, C)$ we have that 
\begin{equation*}
  |C^-| \ge \delta^0(G) - |B| = (n-1)/2 - (|A| + |B|)/2  + (|A| - |B|)/2
= |C|/2 + (|A| - |B|)/2  - 1/2,
\end{equation*}
and a similar lower bound holds for the cardinality of $C^+ = N^+(x^-, C)$.
Let $v \in C^-$ and note that, because no cyclic triangle
contains both $v$ and $x^+$ and has its third vertex in $A$, 
we have that $N^+(v) = A$.
Similarly $N^-(v) = A$ for every $v \in C^+$. 
Therefore, $C^+$ and $C^-$ are disjoint,
so, with the fact that $|A| \ge |B|$, we 
have that $|C|/2 \ge |C^+|, |C^-| \ge |C|/2 - 1/2$ and $|B| \ge |A| - 1$. 

Since $|B|$ is close to $|A|$ and $A$ was the largest part in $\mathcal{P}$,
we can apply similar logic to $G[B]$ to find
$y^+, y^- \in B$ such that $d^+(y^+, B) = d^-(y^-, B) = |B| - 1$
and the sets $N^-(y^+, \overline{B \cup D})$
and $N^+(y^-, \overline{B \cup D})$ almost 
partition $\overline{B \cup D}$ and have roughly the same cardinality.
Let
$A^-= N^-(y^+, A)$, $A^+= N^+(y^-, A)$ and assume that $|A^-|\ge |A^+|$.
(Similar logic gives a contradiction when $|A^+|\ge |A^-|$.)
Since $A \subseteq \overline{B \cup D}$, we have $A^- =A\cap N^-(y^+, \overline{B \cup D})$ and $A^+=A\cap N^+(y^-, \overline{B \cup D})$. Therefore, the sets $A^-$
and $A^+$ almost partition $A$ and then
$|A^-|$ cannot be much smaller than $|A|/2$. Because $G[A^-]$
is a transitive tournament, there exists $x \in A^-$ such that
$d^+(x, A^-) = |A^-| - 1$.
Now consider the out-neighborhood of $x$. 
Recall that $N^-(v) = A$ for every $v \in C^+$, so $N^+(x)$ contains $(A^- - x) \cup C^+$.
Moreover, since $xy^+$ is an edge and there are no triangles with
one vertex in $A$ and two vertices in $B$, the out-neighborhood of $x$ also contains all of $B$.
Note that $|A^- \cup B \cup C^+|$ is roughly $n/2 + |B|/2$,
so, since $|B| \ge \eta n$,
we have contradicted the fact that $G$ is a regular tournament.

\begin{lemma}\label{strong1}
  Suppose that $0 < 1/n \ll \alpha,c \ll \xi, \beta \ll 1$
  and that $G$ is an oriented graph on $n$ vertices such that $\delta^0(G) \ge (1/2 - c)n$.
  If $A, B \subseteq V(G)$ are disjoint and $e^+(A, B) \le \alpha n^2$, then 
  $|\strong{A}{-}{\beta} \cap B| \ge |B| - \xi n$
  and 
  $|\strong{B}{+}{\beta} \cap A| \ge |A| - \xi n$
\end{lemma}
\begin{proof}
  Note that for every $v \in B \setminus \strong{A}{-}{\beta}$, 
  we have that $d^-(v, A) \ge \beta n - (n - 2 \delta^0(G)) \ge \beta n/2$, 
  and also, for every $v \in A \setminus \strong{B}{+}{\beta}$,
  we have that $d^+(v, B) \ge \beta n/2$.
  Therefore, if we let $m$ be the maximum of 
  $|B \setminus \strong{A}{-}{\beta}|$ and $|A \setminus \strong{B}{+}{\beta}|$,
  then $m \cdot \beta n/2 \le e^+(A, B) \le \alpha n^2$,
  which implies that $m \le \xi n$.
\end{proof}

\begin{lemma}\label{degree}
  Suppose that $0 < 1/n \ll c, \alpha \ll \beta < 1$,
  and that $G = (V,E)$ is an $n$-vertex oriented graph such that $\delta^0(G) \ge (1/2 - c)n$.
  If $A \subseteq V$ and $\cyc(A) \le \alpha n^3$, 
  then, for $\sigma \in \{+, -\}$, there exists $x^\sigma \in A$ such that $d^\sigma(x^\sigma, A) \ge |A| - \beta n$.
\end{lemma}
\begin{proof}
  Pick $\xi$ so that 
  \begin{equation*}
    0 < 1/n \ll c, \alpha \ll \xi \ll \beta \ll 1.
  \end{equation*}
  Let $X$ be the set of vertices $x$ in $A$ such that $\cyc(x, A, A) \le \alpha^{1/2} n^2$.
  Since 
  \begin{equation*}
    (|A| - |X|) \cdot \alpha^{1/2} n^2 \le 3\cyc(A) \le 3\alpha n^3
  \end{equation*}
  we have that 
  \begin{equation}\label{eq:sizeofX}
    |X| \ge |A| -  3\alpha^{1/2} n.
  \end{equation}
  Pick $x^+ \in X$ so as to maximize $d^+(x^+, X)$, and
  let $X^+ = N^+(x^+, X)$ and $X^- = N^-(x^+, X)$.
  Note that, because of the minimum semidegree condition, $|X^-| \ge |X| - |X^+| - 2cn$,
  and that $\alpha^{1/2} n^2 \ge \cyc(x^+, A, A) \ge e^+(X^+, X^-)$.
  Therefore, with Lemma~\ref{strong1}, we have that
  \begin{equation*}
    |\strong{X^+}{-}{\xi} \cap X^-| \ge |X^-| - \xi n \ge |X| - |X^+| - 2 c n - \xi n,
  \end{equation*}
  and, with the minimum semidegree condition, there exists 
  $y \in \strong{X^+}{-}{\xi} \cap X^-$ such that
  \begin{equation*} 
    d^+(y, X^-) \ge (|\strong{X^+}{-}{\xi} \cap X^-| - 2c n)/2 \ge (|X| - |X^+| - 2 \xi n)/2
  \end{equation*}
  Because $y \in \strong{X^+}{-}{\xi}$, we have that
  \begin{equation*}
    d^+(y, X^+) \ge |X^+| - \xi n.
  \end{equation*}
  Therefore, by the selection of $x^+$ we have that,
  \begin{equation*}
    |X^+| \ge d^+(y, X^+) + d^+(y, X^-) 
    \ge (|X^+| - \xi n)  + (|X| - |X^+| - 2 \xi n)/2
    = (|X| + |X^+| - 4 \xi n)/2 
  \end{equation*}
  which, with \eqref{eq:sizeofX}, that implies $d^+(x^+, A) = |X^+| \ge |X| - 4 \xi n \ge |A| - \beta n$.
  By a similar argument, we can find $x^- \in A$, such that $d^-(x^-, A) \ge |A| - \beta n$.
\end{proof}

\begin{lemma}\label{lem:Csigma}
  Let $0 < 1/n \ll \alpha, c \ll \xi \ll \eta < 1$.
  Let $G = (V,E)$ be an $n$-vertex oriented graph such that $\delta^0(G) \ge (1/2 - c)n$.
  For every partition $\{A, B, C\}$ of $V$ such that 
  $|A| \ge \eta n$ and $\cyc(A), \cyc(A,A,C) \le \alpha n^3$
  there exist disjoint subsets $C^+$ and $C^-$ of $C$ such that, for $\sigma \in \{+, -\}$,
  \begin{equation*}
    e^{-\sigma}(A, C^\sigma) \le \xi n^2
    \qquad\text{and}\qquad
    |C^\sigma| \ge |C|/2 + (|A| - |B|)/2 - \xi n.
  \end{equation*}
\end{lemma}
\begin{proof}
  We will prove the lemma by showing that, with $\xi/2$ playing the role of $\xi$,
  there exist subsets $C^+$ and $C^-$ of $C$ that meet the conditions of the lemma
  except that $C^+ \cap C^-$ might not be empty.
  This will prove the lemma, because then, by the minimum semidegree condition,
  \begin{equation*}
    (|A| - 2 c n)|C^+ \cap C^-| \le 
    e^{+}(A, C^+ \cap C^-) + e^{-}(A, C^+ \cap C^-) \le \xi n^2,
  \end{equation*}
  and this, with the fact that $|A| \ge \eta n$, implies that $|C^+ \cap C^-| \le \xi n/2$,
  which means that the sets $C^+ \setminus C^-$ and $C^- \setminus C^+$ are disjoint sets
  that meet the conditions of the lemma.
  Furthermore, we will only prove that such a $C^+$ exists, 
  because the existence of the desired set $C^-$ follows by a similar argument.

  Let $X$ be the set of vertices $x$ in $A$ such that $\cyc(x, A, C) \le \alpha^{1/2} n^2$.
  Because 
  \begin{equation*}
    (|A| - |X|) \cdot \alpha^{1/2} n^2 \le 3\cyc(A, A, C) \le 3\alpha n^3 
  \end{equation*}
  we have that $|X| \ge |A| - 3\alpha^{1/2} n$.
  Since $\cyc(X, X, X) \le \cyc(A, A, A) \le \alpha n^3$, Lemma~\ref{degree} implies that there exists $x \in X$ such that
  \begin{equation*}
    d^-(x, A) \ge d^-(x, X) \ge |X| - \xi n/4 \ge |A| - \xi n/3.
  \end{equation*}
  Let $C^+ = N^+(x, C)$. Because every edge directed from $C^+$ to $N^-(x, A)$ corresponds to a triangle in $\cyc(x, A, C)$, 
  we have that 
  \begin{equation*}
    e^-(A, C^+) \le \cyc(x, A, C) + e^-(A \setminus N^-(x, A), C^+) \le \xi n^2/2.
  \end{equation*}
  We also have that 
  \begin{equation*}
  |C^+| \ge \delta^0(G) - d^+(x, A) - |B| \ge n/2 - |B| - cn - \xi n/3 \ge |C|/2 + (|A| - |B|)/2 - \xi n/2.
  \end{equation*}
  This completes the proof of the lemma.
\end{proof}

\begin{proof}[Proof of Lemma~\ref{lem:not-closed-1}]
  Pick $\beta$, $\alpha'$ and $\xi$ so that
  \begin{equation*}
    0 < 1/n \ll c \ll \mu \ll \alpha' \ll \alpha, \xi \ll \beta \ll \eta < 1.
  \end{equation*}
  For a contradiction, assume that 
  \begin{equation}\label{contradict}
    \cyc(V_i, V_i, \overline{V_i}) \ge \alpha n^3 \qquad
    \text{for all $V_i \in \mathcal{P} = \{V_1, \dotsc, V_d\}$}.
  \end{equation}
  Because $\mathcal{P}$ is an $\eta$-partition we have that $1/d \ge \eta$ 
  which implies that
  \begin{equation}\label{main_assumption}
    \forall J \subseteq [d], \forall i \in [d], \exists j' \in J \text{ such that }
    \cyc(V_i, V_i, V_{j'}) \ge \eta \cdot c(V_i, V_i, \bigcup_{j \in J} V_j).
  \end{equation}

  Let $A = V_x$ be the largest part in $\mathcal{P}$.
  By \eqref{contradict} and \eqref{main_assumption} with $J = [d] - x$, 
  there exists $V_y = B \in \mathcal{P} - A$ such that $\cyc(A, A, B) \ge \alpha' n^3$.
  Let $C = V \setminus (A \cup B)$.  
  By \eqref{contradict} and \eqref{main_assumption} with $J = [d] - y$, 
  we also have $V_z = D \in \mathcal{P} - B$ such that
  $c(B, B, D) \ge \alpha' n^3$.  
  Let $F = V \setminus (B \cup D)$.

  Note that 
  $\mathbf{v}_1 = 2\mathbf{u}_x + \mathbf{u}_y \in \uIvec(H(G))$ and
  $\mathbf{v}_2 = 2\mathbf{u}_y + \mathbf{u}_z \in \uIvec(H(G))$.
  If $A = D$, then $\mathbf{v}_1 - \mathbf{v}_2$ is a $2$-transferral in $\ulattice(H(G))$.
  Furthermore, if $\cyc(A) \ge \alpha' n^3$, then $\mathbf{v}_3 = 3\mathbf{u}_x \in \uIvec(H(G))$, 
  so $\mathbf{v}_1 - \mathbf{v}_3$ is a $2$-transferral in $\ulattice(H(G))$.
  Also, if $c(A, A, C) \ge \alpha' n^3$, then, by \eqref{main_assumption} with
  $J = [d] - x - y$, there exists $w \in [d] - x - y$ such that 
  $\mathbf{v}_4 = 2 \mathbf{u}_x + \mathbf{u}_w \in \uIvec(H(G))$, 
  so $\mathbf{v}_1 - \mathbf{v}_4$ is a $2$-transferral in $\ulattice(H(G))$.
  With additional similar arguments, the fact that there are no $2$-transferrals in $\ulattice(H(G))$
  implies that $A \neq D$ and 
  \begin{equation}\label{cc_main_assumption}
    \cyc(A), \cyc(A, A, C), \cyc(B),  \cyc(B, B, F) <  \alpha' n^3.
  \end{equation}
  By Lemma~\ref{lem:Csigma}, there exist disjoint subsets $C^+$, $C^-$ of $C$
  such that, for $\sigma \in \{-, +\}$,
  \begin{equation}\label{ACsigma}
    e^{-\sigma}(A, C^\sigma) \le \xi n^2,
  \end{equation}
  and 
  \begin{equation}\label{Csigmasize1}
    |C^\sigma| \ge |C|/2 + (|A| - |B|)/2 - \xi n.
  \end{equation}
  By the selection of $A$ we have that $|A| \ge |B|$, so \eqref{Csigmasize1} implies that 
  \begin{equation}\label{Csigmasize}
    |C^\sigma| \ge |C|/2 - \xi n.
  \end{equation}
  Because $C^+$ and $C^-$ are disjoint subsets of $C$, 
  we have that $\min\{|C^+|, |C^-|\} \le |C|/2$, so, with \eqref{Csigmasize1},
  \begin{equation}\label{sizeofB}
    |B| \ge |A| - 2 \xi n.
  \end{equation}

  By Lemma~\ref{lem:Csigma}, with $B$ and $F$ playing the roles of $A$ and $C$, respectively, 
  there are disjoint subsets $F^+$ and $F^-$ of $F$
  such that, for $\sigma \in \{-, +\}$,
  \begin{equation}\label{BFsigma}
    e^{-\sigma}(B, F^\sigma) \le \xi n^2.
  \end{equation}
  and 
  \begin{equation*}
    |F^\sigma| \ge |F|/2 + (|B| - |D|)/2 - \xi n.
  \end{equation*}
  By the selection of $A$ and \eqref{sizeofB}, we have that
  \begin{equation*}
    |B| + 2 \xi n \ge |A| \ge |D|,
  \end{equation*}
  so $|F^\sigma| \ge |F|/2 - 2 \xi n$, and
  \begin{equation}\label{Fsigmasize}
    |F^+| + |F^-| \ge |F| - 4 \xi n.
  \end{equation}

  Note that $A \subseteq F$ and fix $\sigma \in \{-, +\}$ so that $|F^{-\sigma} \cap A| \ge |F^{\sigma} \cap A|$.
  By \eqref{Fsigmasize} and the selection of $\sigma$, we have that 
  \begin{equation}\label{FsigmaAsize}
    |F^{-\sigma} \cap A| \ge (|F^+ \cap A| + |F^- \cap A|)/2 \ge (|A| - 4\xi n)/2 = |A|/2 - 2 \xi n. 
  \end{equation}
  Note that Lemma~\ref{strong1}, with \eqref{ACsigma}, implies that 
  \begin{equation*}
    |A \setminus \strong{C^\sigma}{\sigma}{\beta}| \le \beta n/2,
  \end{equation*}
  and Lemma~\ref{strong1}, with \eqref{BFsigma} and \eqref{FsigmaAsize}, implies that
  \begin{equation*}
    |\strong{B}{\sigma}{\beta} \cap F^{-\sigma} \cap A| \ge |A|/2 - \beta n/2. 
  \end{equation*}
  Therefore, if we set 
  \begin{equation*}
    A' = \strong{B}{\sigma}{\beta} \cap \strong{C^\sigma}{\sigma}{\beta} \cap A,  
  \end{equation*}
  then
  \begin{equation*}
    |A'| \ge |\strong{B}{\sigma}{\beta} \cap F^{-\sigma} \cap A| - 
    |A \setminus \strong{C^\sigma}{\sigma}{\beta}|
    \ge |A|/2 - \beta n.
  \end{equation*}
  Since $\cyc(A') \le \cyc(A) \le \alpha n^3$, Lemma~\ref{degree}
  implies that there exists  $x \in A'$ such that 
  \begin{equation*}
    d^\sigma(x, A') \ge |A'| - \beta n \ge |A|/2 - 2 \beta n.
  \end{equation*}
  Recall that $x \in \strong{B}{\sigma}{\beta} \cap \strong{C^\sigma}{\sigma}{\beta}$.
  This implies that
  \begin{equation*}
    d^\sigma(x, B) \ge |B| - \beta n,
  \end{equation*}
  and, with \eqref{Csigmasize}, we have that 
  \begin{equation*}
    d^\sigma(x, C^\sigma) \ge |C^\sigma| - \beta n \ge |C|/2 - 2\beta n.
  \end{equation*}
  Therefore, with the fact that $|B| \ge \eta n$, we have 
  \begin{equation*}
    d^\sigma(x) \ge 
    (|A|/2 - 2\beta n) + (|B| - \beta n) + (|C|/2 - 2 \beta n)
    = n/2 + |B|/2  - 5 \beta n > (1/2 + c)n,
  \end{equation*}
  so $d^{-\sigma}(x) < (1/2 - c) n$, 
  a contradiction to the minimum semidegree condition.
\end{proof}

\subsection{Proof of Lemma~\ref{lem:not-closed-2}}\label{sec:not-closed-2}

We need the following lemma for the proof of Lemma~\ref{lem:not-closed-2}.
Informally, it says that, in an oriented graph $G$ with sufficiently high minimum semidegree,
if $A$ is a reasonably large subset of $V(G)$
for which $\cyc(A, A, \overline{A})$ is $o(n^3)$, then 
essentially half of the vertices in $\overline{A}$ 
have almost all of $A$ as out-neighbors and 
half of the vertices in $\overline{A}$ 
have almost all of $A$ as in-neighbors.
From this, we then argue that $|A|$ cannot be much larger than $n/3$.

At a high-level, we use Lemma~\ref{main-not-closed-2} to prove Lemma~\ref{lem:not-closed-2} 
in the following way
(the actual proof of Lemma~\ref{lem:not-closed-2} appears 
after the proof of Lemma~\ref{main-not-closed-2}). 
We start with a non-trivial 
$\eta$-partition of $G$ such that $\ulattice(H(G))$ is $2$-transferral-free
and such that there exists $A \in \mathcal{P}$ such that $\cyc(A, A, \overline{A})$ is
$o(n^3)$.  We then use Lemma~\ref{main-not-closed-2} to get 
a partition $\{S^+, S^-\}$ of $\overline{A}$  such that $S^+$ and $S^-$ have
roughly the same size and such that $e^+(A, S^-)$ and $e^+(S^+, A)$ are both $o(n^2)$.
We can then finish the proof by showing that $A$, $S^+$ and $S^-$ each have 
roughly the same size
(which is implied if $|S^+|$ is not much more than $n/3$)
and that $e^+(S^-, S^+)$ is $o(n^2)$.
To this end, we use the fact that 
$\ulattice(H(G))$ is $2$-transferral-free to show that 
$\cyc(S^+, S^+, \overline{S^+})$ is $o(n^3)$,
and then apply
Lemma~\ref{main-not-closed-2} again with $S^+$ now playing the role of $A$.
\begin{lemma}\label{main-not-closed-2}
  Suppose that 
  \begin{equation*}
    1/n \ll c \ll \alpha \ll \beta \ll \xi \ll \eta < 1,
  \end{equation*}
  and $G = (V, E)$ is an oriented graph such that $\delta^0(G) \ge (1/2 - c)n$.
  For every $A \subseteq V$ such that $|A| > \eta n$ and $\cyc(A, A, \overline{A}) \le \alpha n^3$,
  we have that 
  \begin{itemize}
    \item $|\SN^+_\beta(A)|, |\SN^-_\beta(A)| = |\overline{A}|/2 \pm \xi n$, and 
    \item $|A| \le (1/3 + \xi)n$.
  \end{itemize}
\end{lemma}
\begin{proof}
  Let $S^+ = \SN^+_\beta(A)$, $S^- = \SN^-_\beta(A)$ and $S = V \setminus (S^+ \cup S^- \cup A)$,
  so $\{S, S^+, S^-, A\}$ is a partition of $V$.
  We will first prove that 
  \begin{equation}\label{Ssmall}
    |S| \le \beta n.
  \end{equation}
  To this end, let $x \in S \subseteq \overline{A}$ 
  and note that $x$ must have at least $\beta n - 2c n \ge \beta n/2$ out-neighbors and in-neighbors in $A$. 
  Therefore, 
  \begin{equation}\label{initial_bound}
    \begin{split}
      |S| (\beta n)^2/4 &\le \sum_{x \in S} d^+(x, A) \cdot d^-(x, A) \le \sum_{x \in \overline{A}} d^+(x, A) \cdot d^-(x, A) \\ 
      &= \sum_{xy \in E(A)} d^{-,+}(xy,\overline{A}) + d^{+,-}(xy, \overline{A}) 
      = \cyc(A, A, \overline{A}) + \sum_{xy \in E(A)} d^{+,-}(xy, \overline{A}).
    \end{split}
  \end{equation}
  By Lemma~\ref{p0}, we have that
  \begin{equation}\label{second_bound}
    \begin{split}
      \sum_{xy \in E(A)} d^{+,-}(xy, \overline{A})  &=
    \sum_{xy \in E(A)} \left(d^{+,-}(xy) - d^{+,-}(xy, A)\right) \\
    &\le \sum_{xy \in E(A)} \left(d^{-,+}(xy) - d^{+,-}(xy, A) + 4cn\right) \\
    &= \sum_{xy \in E(A)} \left(d^{-,+}(xy, \overline{A}) + d^{-,+}(xy, A) - d^{+,-}(xy, A) + 4cn\right)  \\
    &= \cyc(A, A, \overline{A}) + 3\cyc(A) -\trn(A) + 4cn \cdot e(A).
  \end{split}
  \end{equation}
  Because $e(A)/|A| = \sum_{v \in A} d^+(v, A)/|A| \ge \frac{|A| - 2cn}{2}$ and $f(x) = \binom{x}{2}$ is convex, we have that 
  \begin{equation*}
    \trn(A) = \sum_{v \in A}{\binom{d^+(v,A)}{2}} \ge |A|\binom{e(A)/|A|}{2} 
    \ge \frac{|A|(|A| - 2cn - 1)^2}{8} \ge \frac{|A|^3}{8} - c n^3
  \end{equation*}
  Therefore, since $\cyc(A) + \trn(A) \le \binom{|A|}{3}$,
  \begin{equation}\label{third_bound}
    3\cyc(A) - \trn(A) \le 3 \binom{|A|}{3} - 4 \trn(A) \le 4 c n^3
  \end{equation}
  Combining \eqref{initial_bound}, \eqref{second_bound} and \eqref{third_bound}, we have that
  \begin{equation*}
    |S| (\beta n)^2/4 \le 2 \cyc(A, A, \overline{A}) + 8 cn^3 \le  3 \alpha n^3,
  \end{equation*}
  and \eqref{Ssmall} holds.

  For $\sigma \in \{+, -\}$, we have that
  $e^{\sigma}(A, \overline{A}) \ge (|A| - \beta n)\cdot|S^\sigma|$
  and, by \eqref{Ssmall} and the definition of $S^{-\sigma}$,
  \begin{equation*}
    e^{\sigma}(A, \overline{A}) \le |A|\cdot|S^{\sigma}| + (|A| - \beta n)\cdot|S| +  \beta n \cdot |S^{-\sigma}|
    \le |A|\cdot|S^{\sigma}| + 2 \beta n^2,
  \end{equation*}
  so $|A| \cdot |S^\sigma| = e^\sigma(A, \overline{A}) \pm 2 \beta n^2$.
  With Lemma~\ref{in=out}, we have that
  \begin{equation*}
    |A| \cdot |S^+| = e^+(A, \overline{A}) \pm 2 \beta n^2 = e^-(A, \overline{A}) \pm 4 \beta n^2 = |A| \cdot |S^-| \pm 6 \beta n^2.
  \end{equation*}
  Therefore, because $|A| \ge \eta n$, we have that $|S^+| = |S^-| \pm \xi n$.
  This with \eqref{Ssmall} implies that
  \begin{equation*}
    |\overline{A}| = n - |A| = |S^+| + |S^-| + |S| = 2 |S^{\sigma}| \pm 2\xi n,
  \end{equation*}
  which proves the first part of the lemma.

  We will prove the second part of the lemma by contradiction,
  so assume that $|A| > (1/3 +  \xi)n$. 
  There exists $x \in S^+$ such that $d^+(x, S^+) < |S^+|/2$.
  Using the fact that
  $|\overline{A}| < (2/3 - \xi)n$, and,
  by the first part of the lemma,
  $|S^+| \ge |\overline{A}|/2 - \xi n$,
  we have that
  \begin{equation*}
    d^+(x, \overline{A}) < |S| + |S^-| + |S^+|/2 =
    |\overline{A}| - |S^+|/2 \le 
    |\overline{A}| - |\overline{A}|/4 + \xi n/2 < n/2  - \xi n/4.
  \end{equation*}
  Because $x \in S^+$, we have that $d^+(x, A) \le \beta n$, so
  \begin{equation*}
    d^+(x) = d^+(x, A) + d^+(x, \overline{A}) < (1/2 - c)n,
  \end{equation*}
  a contradiction to the minimum semidegree condition.
\end{proof}

\begin{proof}[Proof of Lemma~\ref{lem:not-closed-2}] 
  Define $\xi$, $\beta$ and $\omega$ so that 
  \begin{equation*}
  0 < 1/n \ll c \ll \mu, \alpha \ll \xi \ll \beta \ll \omega \ll \gamma, \eta < 1.
  \end{equation*}
  By assumption, there exists
  $\mathcal{P}$ a non-trivial $\eta$-partition of $V(G)$ such that $\ulattice(H(G))$ is $2$-transferral-free
  and $A \in \mathcal{P}$ such that $\cyc(A, A, \overline{A}) \le \alpha  n^3$.
  Since $\mathcal{P}$ is an $\eta$-partition, we have that
  \begin{equation}\label{Alarge}
    |A| \ge \eta n.
  \end{equation}

  For subsets $U_1, U_2, U_3, U_4$ of $V$, 
  let $L(U_1, U_2, U_3, U_4)$ be the collection of $4$-sets
  $\{u_1, u_2, u_3, u_4\}$ such that $u_i \in U_i$ for $i \in [4]$
  and both $u_1u_2u_3$ and $u_2u_3u_4$ are cyclic triangles.
  We can assume that
  \begin{equation}\label{lAVVAbar-small}
    |L(A, V, V, \overline{A})| \le \xi n^4, 
  \end{equation}
  because otherwise, since $|\mathcal{P}| \le 1/\eta$, 
  there would exist $B, C \in \mathcal{P}$ and $D \in \mathcal{P} - A$ 
  such that $|L(A, B, C, D)| \ge \eta^3 \cdot \xi n^4 \ge \mu n^4$, which would in turn imply that
  both $c(A, B, C) \ge \mu n^3$ and $c(B, C, D) \ge \mu n^3$, which contradicts
  the fact that $\ulattice(H(G))$ is $2$-transferral-free.

  Define $S^\sigma = \strong{A}{\sigma}{\beta}$ for $\sigma \in \{+, -\}$, 
  and let $S = V \setminus (S^+ \cup S^- \cup A)$.
  Lemma~\ref{main-not-closed-2} implies that 
  \begin{equation}\label{sizeofZ}
    |S| \le 2\xi n,
  \end{equation}
  and that 
  \begin{equation}\label{sizeofA}
    |A| \le (1/3 + \xi)n.
  \end{equation}
  For $\sigma \in \{+, -\}$, suppose that there exists 
  $xyz$ a triangle with $xy \in E(S^\sigma)$ and $z \in S^{-\sigma}$,
  i.e., $xyz$ is one of the triangle that is counted in $\cyc(S^\sigma, S^\sigma, S^{-\sigma})$.
  If $\sigma = +$, then for every $w \in N^{-,+}(yz, A)$, we have $\{w,y,z,x\} \in L(A, S^+, S^-, S^+)$,
  and if $\sigma = -$, then for every $w \in N^{-,+}(zx, A)$, we have $\{w,x,z,y\} \in L(A, S^-, S^+, S^-)$.
  Therefore, because $d^{-\sigma}(x, A), d^{-\sigma}(y, A), d^{\sigma}(z, A) \ge |A| - \beta n$,
  \eqref{lAVVAbar-small} implies that
  \begin{equation*}
    (|A| - 2 \beta n) \cdot c(S^\sigma, S^\sigma, S^{-\sigma})
    \le |L(A, S^\sigma, S^{-\sigma}, S^\sigma)| 
    \le \xi n^4. 
  \end{equation*}
  Therefore, with \eqref{Alarge}, we have
  \begin{equation*}
    \cyc(S^\sigma, S^\sigma, S^{-\sigma}) \le \beta n^3.
  \end{equation*}
  Furthermore,
  \begin{equation*}
    \cyc(S^\sigma, S^\sigma, A) \le 2\beta n \cdot \binom{|S^\sigma|}{2} \le \beta n^3, 
  \end{equation*}
  and, because \eqref{sizeofZ} implies that
  \begin{equation*}
    \cyc(S^\sigma, S^\sigma, S) \le 2\beta n^3, 
  \end{equation*}
  we can deduce that 
  \begin{equation*}
    \cyc(S^\sigma, S^\sigma, \overline{S^\sigma}) = \cyc(S^\sigma, S^\sigma, S^{-\sigma}) + \cyc(S^\sigma, S^\sigma, A) + \cyc(S^\sigma, S^\sigma, S) 
    \le 4 \beta n^3.
  \end{equation*}
  Applying Lemma~\ref{main-not-closed-2} with $S^\sigma$, $4 \beta$ and $\omega$ playing the roles of $A$, $\alpha$ and $\xi$, respectively, 
  implies that 
  \begin{equation}\label{ubAsigma}
    |S^+|, |S^-| \le (1/3 + \omega)n.
  \end{equation}

  Let $A' = A \cup S$ and note that $\{A', S^+, S^-\}$ is a partition of $V(G)$.
  By \eqref{sizeofZ}, we have that 
  \begin{equation*}
    d^{-\sigma}(v, A') \ge |A| - \beta n \ge |A'| - \omega n \text{ for every $v \in S^{\sigma}$}.
  \end{equation*}
  which implies that 
  \begin{equation}\label{AsigmaV1edges}
    e^{\sigma}(S^{\sigma}, A') \le |S^{\sigma}| \cdot \omega n \le \omega n^2, 
  \end{equation}
  and, also, with Lemma~\ref{in=out}, that
  \begin{equation*}
    e^+(S^+, \overline{S^{+}}) \ge e^+(\overline{S^+}, S^+) - cn^2 \ge e^+(A', S^{+}) \ge (|A'| - \omega n)|S^{+}|.
  \end{equation*}
  Combining this with \eqref{sizeofZ}, \eqref{sizeofA}, \eqref{ubAsigma} and \eqref{AsigmaV1edges} gives us that 
  \begin{equation*}
    e^+(S^+, S^-) = e^+(S^+, \overline{S^+}) - e^+(S^+, A') \ge (|A'| - \omega n)|S^+| - \omega n^2 \ge (|S^-| - 3\omega n)|S^+|,
  \end{equation*}
  so, with \eqref{AsigmaV1edges}, we have that 
  \begin{equation*}
    e^+(S^-, S^+), e^+(S^+, A'), e^+(A', S^-) \le \gamma n^2.
  \end{equation*}
  Also, \eqref{sizeofZ}, \eqref{sizeofA} and \eqref{ubAsigma} imply that $|A'|, |S^+|, |S^-| \le n/3 + \omega n$, so
  \begin{equation*}
    |A'|, |S^+|, |S^-| = n/3 \pm 2\omega n = n/3 \pm \gamma n,
  \end{equation*}
  and $G$ is $\gamma$-extremal.
\end{proof}

\subsection{Proof of Theorem~\ref{lem:extremal} - The Extremal Case}\label{sec:extremal}
We use the following theorem from \cite{johansson2000triangle}, 
although we do not require its full strength.  In fact, we only use it for 
corollary which follows that could easily be proved directly.
\begin{thm}[Johansson 2000 \cite{johansson2000triangle}]\label{balanceR}
  Let $G = (V, E)$ be a graph.  If there exists $\{V_1, V_2, V_3\}$ a partition of $V$
  such that $|V_1| = |V_2| = |V_3| = m$ and, for every $i \in [3]$ and $v \in V_i$, 
  $|N(v) \cap V_{i+1}|$, $|N(v) \cap V_{i-1}| \ge \frac23m+\sqrt{m}$ then $G$ has a triangle factor.
\end{thm}

\begin{cor}\label{balance}
  Suppose that $0< 1/n \ll c, \xi \ll 1$, $n$ is a multiple of $3$, and 
  $G = (V,E)$ is an oriented graph on $n$ vertices such that $\delta^0(G) \ge \left(\frac12 - c\right)n$.
  If there exists a partition $\{V_1, V_2, V_3\}$ of $V$ such that 
  $|V_1|\equiv|V_2|\equiv|V_3| \pmod 3$ and, for every $i\in [3]$ and $v\in V_i$, 
  we have $d^{+}(v, V_{i+1})$, $d^{-}(v, V_{i-1}) \geq \left(\frac13-\xi\right)n$, then $G$ has a cyclic triangle factor.
\end{cor}
\begin{proof}
  Assume without loss of generality that $|V_1| \le |V_2| \le |V_3|$.
  For every $i \in [3]$, the degree condition implies that 
  $|V_i| = (1/3 \pm 2\xi)n$, so 
  \begin{equation*}
    \delta^0(G[V_i]) \ge (1/2 - c) n - 3\xi - (1/3 + 2 \xi)n \ge (1/2 - 4c - 16\xi)|V_i|. 
  \end{equation*}
  Therefore, by Lemma~\ref{min_vertex_degree}, we can greedily find
  a collection $\mathcal{C}_2$ of $(|V_2| - |V_1|)/3 \le 2 \xi n$ vertex-disjoint cyclic-triangles in $G[V_2]$, and 
  a collection $\mathcal{C}_3$ of $(|V_3| - |V_1|)/3 \le 2 \xi n $ vertex-disjoint cyclic-triangles in $G[V_3]$.
  Then, Theorem~\ref{balanceR} implies that there exists a cyclic triangle factor $\mathcal{C}$ of
  the oriented graph induced by $V \setminus V(\mathcal{C}_2 \cup \mathcal{C}_3)$, and
  $\mathcal{C} \cup \mathcal{C}_2 \cup \mathcal{C}_3$ is a triangle factor of $G$.
\end{proof}

\begin{proof}[Proof of Theorem~\ref{lem:extremal}]
  Since $G$ is $\gamma$-extremal, there exists a partition $\{V_1, V_2, V_3\}$ of $V$
  such that for every $i\in[3]$, $|V_i| =(1/3 \pm \gamma)n$ and $e^-(V_{i}, V_{i+1}) \leq \gamma n^2$.
  Introduce new constants $\alpha$ and $\beta$, so that
  \begin{equation*}
    0 < c, \gamma \ll \alpha \ll \beta \ll 1.
  \end{equation*}
  Let
  \begin{equation*}
    U_i = \strong{V_{i+1}}{-}{\beta} \cap \strong{V_{i-1}}{+}{\beta} \qquad \text{ for $i \in [3]$,}
  \end{equation*}
  and let $U_0 = V \setminus (U_1 \cup U_2 \cup U_3)$.
  Note that $\{U_0, U_1, U_2, U_3\}$ is a partition of $V$.
  For every $i \in [3]$ and every $v \in V_i \setminus U_i$, 
  at least one of $d^-(v, V_{i+1})$ or $d^+(v, V_{i-1})$ is at least $\beta n - 2cn \ge \beta n / 2$, so 
  \begin{equation*}
    \frac{1}{2} \sum_{i = 1}^{3} |V_i \setminus U_i| \cdot \beta n 
    \le 2e^-(V_1, V_2) + 2e^-(V_2, V_3) + 2e^-(V_3, V_1) \le 6 \gamma n^2,
  \end{equation*}
  and we have that 
  \begin{equation}\label{sizeofU0}
    |U_0| \le \sum_{i=1}^{3} |V_i \setminus U_i| \le \alpha n.
  \end{equation}

  Suppose that $|U_0| \ge 1$.
  By \eqref{sizeofU0} and Lemma~\ref{min_vertex_degree}, 
  we can greedily find a (possibly empty) collection $\mathcal{C}$ of vertex-disjoint cyclic-triangles such that 
  \begin{equation}\label{sizeofVC}
    |V(\mathcal{C})| \le 3(|U_0| - 1) \le 3 \alpha n
  \end{equation}
  and $|U_0 \cap V(\mathcal{C})| = |U_0| - 1$.
  Let $u$ be the vertex in $U_0$ that is not covered by $\mathcal{C}$, i.e., $\{u\} = U_0 \setminus  V(\mathcal{C})$.
  Let $G' = G - V(\mathcal{C})$, and, for $i \in [3]$, let $W_i = U_i \setminus V(\mathcal{C})$.
  Note that, by \eqref{sizeofU0} and \eqref{sizeofVC},
  \begin{equation}\label{sizeofWi}
    |V_i \setminus W_i| \le 4 \alpha n \qquad \text{ for $i \in [3]$}.
  \end{equation}
  Because $|G'| = |V \setminus V(\mathcal{C})|$ is divisible by $3$, we have that $|W_1| + |W_2| + |W_3| \equiv 2 \pmod 3$, 
  and we can assume without loss of generality that $|W_2| \equiv |W_3| \pmod 3$.
  Fix $x \in \{0, 1, 2\}$ so that
  \begin{equation*}
    |W_1| \equiv x + 2 \text{ and } |W_2| = |W_3| \equiv x \pmod 3.
  \end{equation*}
  If there exists a triangle $T$ in $G'$ such that $u \in V(T)$
  and either $|V(T) \cap W_1| = 2$ or $|V(T) \cap W_2| = |V(T) \cap W_3| = 1$, then 
  we can complete a cyclic triangle factor of $G$ using Corollary~\ref{balance}, so assume the contrary.
  
  Let $W^+_1 = N^+(u, W_1)$, $W^-_1 = N^-(u, W_1)$ and $W_1^0 = W_1 \setminus N(u)$.
  Note that $e^+(W^+_1, W^-_1) = 0$ by assumption,
  and that by the minimum semidegree condition,
  \begin{equation}\label{sizeofW10}
    |W_1^0| \le 2cn.
  \end{equation}
Assume that $|W^-_1| \le |W^+_1|$. If $W^-_1 \neq \emptyset$, then there exists $x \in W^-_1$ such that 
\begin{equation*}
    d^-(x, W^-_1) < |W^-_1|/2 \le |W_1|/4,
 \end{equation*}
 so, by \eqref{sizeofWi} and \eqref{sizeofW10},
   \begin{equation*}
    d^-(x, W^+_1 \cap V_1) \ge \delta^0(G) - \left(d^-(x, W^-_1 \cap V_1) + |W^0_1 \cap V_1| + |V_1 \setminus W_1| + d^-(x, V_2) + |V_3| \right) > 0,
  \end{equation*}
a contradiction. Therefore, we have $W^-_1 = \emptyset$.
  This, with \eqref{sizeofWi} and \eqref{sizeofW10}, implies that
  \begin{equation}\label{woutU1}
    d^+(u, V_1) \ge |W_1 \cap V_1| - |W^0_1| = |V_1| - |V_1 \setminus W_1| - |W^0_1| \ge |V_1| - \beta n,
  \end{equation}
  and
  \begin{equation*}
    d^-(u, W_3 \cap V_3) \ge \delta^0(G) - \left(d^-(u, V_1) + |V_2| + |V_3 \setminus W_3| \right) > (1/6 - 2\beta)n.
  \end{equation*} 
  By the definition of $W_2$, this implies that every vertex in $W_2$ has an 
  out-neighbor in $N^-(u, W_3 \cap V_3)$.
  Therefore, because $e^+(N^+(u, W_2), N^-(u, W_3)) = 0$ by assumption,
  we have that $N^+(u, W_2) = \emptyset$, and,
  by a computation similar to \eqref{woutU1} 
  \begin{equation*}
    d^-(u, V_2) \ge |W_2 \cap V_2| - |W_2 \setminus N(u)| = |V_2| - |V_2 \setminus W_2| - |W_2 \setminus N(u)| \ge |V_2| - \beta n,
  \end{equation*}
  and this with \eqref{woutU1} implies that $u \in U_3$, a contradiction to the fact that $u \in U_0$.
  By a similar argument, if $|W^-_1| \ge |W^+_1|$, one can show that $W_1^+=\emptyset$ and then $u \in U_2$, a contradiction.

  Now suppose that $|U_0| = 0$.
  By Lemma~\ref{balance}, we are done if $|U_1| \equiv |U_2| \equiv |U_3| \pmod 3$,
  so assume the contrary.
  If $G$ has no divisibility barriers,
  then $|V|$ is divisible by $3$, and,
  without loss of generality we can assume that 
  there exists an edge $ab$ such that $a \in U_2$ and $b \in U_1$.
  Because $|V|$ is divisible by $3$,
  we can fix $x \in \{0, 1, 2\}$ and $y \in \{1, 2\}$ so that 
  \begin{equation*}
    \text{$|U_1| \equiv x$, $|U_2| \equiv x + y$,  and $|U_3| = x - y \pmod 3$}
  \end{equation*}
  Note that 
  \begin{equation*}
    d^+(b, U_1) \ge \delta^0(G) - (|V_1 \setminus U_1| + |V_2| + d^+(b, V_3)) \ge (1/6 - 2\beta)n,
  \end{equation*}
  and a similar bound holds for $d^-(a, U_2)$, so
  there exists a vertex $c \in U_2 \cap N^+(b) \cap N^-(a)$, 
  and a vertex $d \in U_1 \cap N^+(b) \cap N^-(a)$. 
  Hence $T_1 = abc$ and $T_2 = abd$ are cyclic triangles
  and, for $j \in \{1, 2\}$,
  \begin{equation*}
    \text{$|V(T_j) \cap U_1| \equiv j$ and $|V(T_j) \cap U_2| \equiv 2j \pmod 3$}.
  \end{equation*}
  Therefore, Lemma~\ref{balance} implies that there exists a cyclic triangle factor 
  $\mathcal{C}$ of $G - T_y$, and
  $\mathcal{C} \cup \{T_y\}$ is a cyclic triangle factor of $G$.
\end{proof}

\hbadness 10000\relax
\bibliographystyle{amsplain}
\providecommand{\bysame}{\leavevmode\hbox to3em{\hrulefill}\thinspace}
\providecommand{\MR}{\relax\ifhmode\unskip\space\fi MR }
\providecommand{\MRhref}[2]{%
  \href{http://www.ams.org/mathscinet-getitem?mr=#1}{#2}
}
\providecommand{\href}[2]{#2}

\end{document}